\documentclass[10pt,a4paper]{scrartcl}

\usepackage[english]{babel}
\usepackage{abstract}
\usepackage{amsmath}
\usepackage{amsfonts}
\usepackage{amssymb}
\usepackage{mathtools}
\usepackage{enumerate}
\usepackage[svgnames]{xcolor}
\usepackage[amsmath, amsthm, thmmarks, framed]{ntheorem}
\usepackage{comment}
\usepackage{tikz}
\usetikzlibrary{plotmarks}
\usepackage[textsize=footnotesize]{todonotes}

\newtheorem{theorem}{Theorem}[section]

\newtheorem{proposition}[theorem]{Proposition}
\newtheorem{lemma}[theorem]{Lemma}
\theoremstyle{remark}
\newtheorem{remark}[theorem]{Remark}

\makeatletter
\def\blfootnote{\xdef\@thefnmark{}\@footnotetext}
\makeatother

\providecommand{\keywords}[1]{\textbf{Keywords: } #1}

\providecommand{\subjclass}[1]{\textbf{Mathematics Subject Classification 2010: } #1}

\DeclarePairedDelimiterX\MeijerM[3]{\lparen}{\rparen}{\begin{matrix}#1 \\ #2\end{matrix}\delimsize\vert\,#3}

\DeclareMathOperator{\Tr}{Tr}

\numberwithin{equation}{section}

\title{Jacobi polynomial moments and products of random matrices}
\author{Wolfgang Gawronski\thanks{Department of Mathematics, University of Trier, 54286 Trier, Germany. E-mail: Gawron@uni-trier.de}, Thorsten Neuschel, Dries Stivigny \thanks{Department of Mathematics, KU Leuven, Celestijnenlaan 200B box 2400, BE-3001 Leuven, Belgium. E-mail: Thorsten.Neuschel@wis.kuleuven.be, Dries.Stivigny@wis.kuleuven.be}}
\date{\today}

\begin{document}

\maketitle


\begin{abstract}Motivated by recent results in random matrix theory we will study the distributions arising from products of complex Gaussian random matrices and truncations of Haar distributed unitary matrices. We introduce an appropriately general class of measures and characterize them by their moments essentially given by specific Jacobi polynomials with varying parameters. Solving this moment problem requires a study of the Riemann surfaces associated to a class of algebraic equations. The connection to random matrix theory is then established using methods from free probability.
\end{abstract}

\keywords{Moment problem; Jacobi polynomials; Raney distributions; Random matrices; Distribution of eigenvalues; Free probability theory; Free multiplicative convolution}

\vspace*{0.3cm}
\subjclass{30E05 ; 15B52 , 30F10 , 46L54}

\section{Introduction}

Products of random matrices are subject to research for many years now. It dates back to the 1960's when Furstenberg and Kesten \cite{Furstenberg_Kesten} studied products of random matrices letting the number of factors grow to infinity while keeping the dimension fixed. This work was very influential and had applications to, for example, Sch\"odinger operator theory \cite{Bougerol_Lacroix}. A more recent development is the study of the distribution of the eigenvalues and (squared) singular values of products consisting of a fixed number of factors as the dimensions grow to infinity. Different approaches have been found, e.g. free probability theory, to obtain the so-called limiting global eigenvalue or (squared) singular value distribution \cite{Alexeev_Gotze, Burda_Janik_Waclaw, Burda_1, Burda_extended, Dupic, Gotze_Tikhomirov, ORourke_Soshnikov, Penson_Z}. In particular, the product of independent standard complex Gaussian matrices (these matrices are called Ginibre random matrices) has attracted interest with applications in, e.g., so-called multiple-input and multiple-output (MIMO) communication networks \cite{Akemann_Kieburg_Wei} (see \cite{Tulino_Verdu} for a more general introduction). In this context, it is also of interest to consider products involving Ginibre matrices and other random matrices. In \cite{Forrester} this was done for the product of Ginibre matrices and inverses of Ginibre matrices and in \cite{Kuijlaars_Stivigny} for the product of Ginibre matrices and truncations of unitary matrices (for applications, see, e.g., the introduction of \cite{Akemann_Nagao}). 

Let $r,s \in \mathbb{N} = \{0, 1, 2, \ldots\}$ with $s < r$ and let $T_1, \ldots, T_s$ be $s$ independent truncations of Haar distributed unitary matrices (such a matrix $T_j$ can be considered to be the upper left block of a Haar distributed unitary matrix). Moreover, let $G_{s+1}, \ldots, G_r$ be $r-s$ independent Ginibre random matrices. The motivation of this paper is to characterize the limiting distributions of the squared singular values of the product
\begin{equation*}
Y_{r,s} := G_r \ldots G_{s+1} T_s \ldots T_1.
\end{equation*}
This is equivalent with studying the limiting eigenvalue distribution of the Wishart-type matrix $Y_{r,s}^{\ast}Y_{r,s}$. The case $s = 0$, where we only have Ginibre random matrices, has been studied in \cite{Penson_Z, Neuschel}. The limiting distribution was shown to be characterized by its moments
\begin{equation*}
FC_r(n) := \frac{1}{rn+1} {rn+n \choose n}, \qquad n \in \mathbb{N},
\end{equation*}
for fixed $r$. These numbers are called Fuss-Catalan numbers of order $r$ and historically arose in the context of combinatorial problems \cite{Knuth}. We will denote the corresponding distributions by $FC_r$. In case $s=1$, it turns out that the limiting distribution of the squared singular values of $Y_{r,1}$ coincides with a specific Raney distribution (Theorem \ref{thm: connection_rmt}, Remark \ref{remark: lim_distr_case_s_1} and also \cite{MNPZ, Neuschel_Stivigny}). These distributions are a generalization of the Fuss-Catalan distributions and are defined by their moments, the so-called Raney numbers
\begin{equation*}
R_{\alpha, \beta}(n) := \frac{\beta}{n\alpha + \beta} {n\alpha + \beta \choose n}, \qquad n \in \mathbb{N},
\end{equation*}
for given $\alpha, \beta \in \mathbb{N}$ such that $\alpha > 1$ and $0 \leq \beta \leq \alpha$. These numbers have a combinatorial interpretation as well, see \cite{Forrester_Liu} for an overview. We will denote the corresponding distributions by $R_{\alpha, \beta}$. One can easily see that $R_{r+1, 1}(n) = FC_r(n)$ and so in the cases $s=0$ and $s=1$ the above mentioned limiting distributions of squared singular values are contained in the class of Raney distributions. However, for $s > 1$ the limiting distribution turns out not to belong to this class anymore. The main goal of this paper is to introduce and characterize an appropriately general class of measures that contains all these limiting distributions. In the language  of free probability theory, this means we want to characterize a class of measures containing all multiplicative free convolutions (see Section 3.1) of the form
\begin{equation*} 
FC_{r-s} \boxtimes R_{1, \frac{1}{2}}^{\boxtimes s}.
\end{equation*}

With this in mind, in Section \ref{sec: jac_moments} we introduce a sequence, depending on $a > 0$ and $r,s \in \mathbb{N}$ such that $s < r$, of positive numbers
\begin{equation} \label{eq: def_mu_0}
J_{r,s,a}(0) := a
\end{equation}
and for $n \in \mathbb{N}$ and $n \geq 1$
\begin{equation} \label{eq: def_mu_n}
J_{r,s,a}(n) := \frac{a}{n} \left(\frac{a^r}{(1+a)^s}\right)^n P_{n-1}^{(\alpha_{n-1}, \beta_{n-1})}\left(\frac{1-a}{1+a}\right),
\end{equation}
where $P_n^{(\alpha_n, \beta_n)}(x)$ are the Jacobi polynomials with varying parameters $\alpha_n = rn + r + 1$ and $\beta_n = -(r+1-s)n - (r+2-s)$ as defined in \cite{Szego_Orth_Pol}. We will prove in Section 2 that these numbers indeed form a (Hausdorff) moment sequence of a compactly supported measure $J_{r,s,a}$. More precisely, we prove the following theorem.
\begin{theorem} 
Let $r,s \in \mathbb{N}$ such that $s < r$ and let $a$ be a positive real number. Then there exists a unique measure $J_{r,s,a}$ on $[0, x^{\ast}]$ with total mass $a$ such that the moments are given by the numbers \eqref{eq: def_mu_0} and \eqref{eq: def_mu_n}. 
\end{theorem}
The right endpoint \(x^{\ast}\) of the support of $J_{r,s,a}$ is defined below in \eqref{eq: x_ast}. The proof of this result heavily relies on a study of the Riemann surface associated to the algebraic equation
\begin{equation*}
w^{r+1} - x(w-a)(w+1)^s = 0
\end{equation*}
which is done in Proposition \ref{prop: study_alg_eq}. 

In Section \ref{sec: appl_rmt} we establish the connection with random matrix theory and we prove in Theorem \ref{thm: connection_rmt} that for $s < r$ we have
\begin{equation*}
J_{r,s,1} = FC_{r-s} \boxtimes R_{1, \frac{1}{2}}^{\boxtimes s}.
\end{equation*}
In particular, we can identify $J_{r,s,1}$ in the case $s = 0$ with the Fuss-Catalan distribution $FC_r$ and in the case $s = 1$ with the Raney distribution $R_{\frac{r+1}{2}, \frac{1}{2}}$. 

Finally, we want to emphasize the remarkable fact that the combination of Theorem \ref{thm: measure_mu} and Theorem \ref{thm: connection_rmt} establishes a further connection between random matrix theory and the theory of classical orthogonal polynomials.

\section{Jacobi polynomial moments} \label{sec: jac_moments}

We start by showing that $J_{r,s,a}(n)$ is a positive-valued sequence. 
\begin{proposition} \label{prop: positivity_mu_n}
Let $a > 0$ be a positive real number and $r,s$ positive integers such that $s \leq r$. Then $J_{r,s,a}(n) > 0$ for all $n \in \mathbb{N}$ where $J_{r,s,a}(n)$ is given by \eqref{eq: def_mu_n}. 
\end{proposition}
To prove this, we need the following lemma. 
\begin{lemma} \label{lemma: mu_n_id_derivative}
Let $a > 0$ be a positive real number and $r,s$ positive integers such that $s \leq r$. Then
\begin{equation} \label{eq: mu_n_id_derivative}
J_{r,s,a}(n) = \frac{1}{n!}\frac{d^{n-1}}{dz^{n-1}} \left.\left(\frac{z^{n(r+1)}}{(1+z)^{ns}}\right)\right|_{z = a}
\end{equation}
for all $n \in \mathbb{N}$ and $n \geq 1$.
\end{lemma}
\begin{proof}
This follows from Leibniz' rule. Indeed, we know that
\begin{align*}
\frac{d^{n-1}}{dz^{n-1}} \left(\frac{z^{n(r+1)}}{(1+z)^{ns}}\right) &= \sum_{k = 0}^{n-1} {n-1 \choose k} \frac{d^{n-1-k}}{dz^{n-1-k}}\left(z^{n(r+1)}\right) \frac{d^k}{dz^k}\left((1+z)^{-ns}\right) \\
&= \sum_{k = 0}^{n-1} {n-1 \choose k} (nr + k + 2)_{n-1-k} z^{nr +k + 1} (-ns - k +1)_k (1+z)^{-ns-k} \\
&= \frac{z^{nr + 1}}{(1+z)^{ns}} P_{n-1}\left(\frac{z}{1+z}\right)
\end{align*}
where \((a)_k\) denotes the Pochhammer symbol and
\begin{equation*}
P_n(x) := \sum_{k = 0}^n {n \choose k} ((n+1)r + k + 2)_{n-k} (-(n+1)s - k +1)_k x^k.
\end{equation*}
Using the representation (see, e.g., \cite{Szego_Orth_Pol}, p.62) 
\[P_n^{(\alpha_n, \beta_n)}(z)=\frac{1}{n!} \sum_{k=0}^n\binom{n}{k} (n+\alpha+\beta+1)_k (\alpha+k+1)_{n-k} \left(\frac{z-1}{2}\right)^k,\]
it is now straightforward to check that
\begin{equation*}
P_n(x) = n! P_n^{(\alpha_n, \beta_n)}(1-2x)
\end{equation*}
with $\alpha_n = rn + r + 1$ and $\beta_n = -(r+1 - s)n - (r+2 - s)$. 
\end{proof}

\begin{proof}[Proof of Proposition \ref{prop: positivity_mu_n}]
This is clearly true for $n = 0$, so let $n \in \mathbb{N}$ and $n \geq 1$. We claim that
\begin{equation*}
\frac{d^{k}}{dz^{k}} \left.\left(\frac{z^{n(r+1)}}{(1+z)^{ns}}\right)\right|_{z = x} > 0, \quad k = 0, \ldots, n, \quad x > 0
\end{equation*}
from which the statement then immediately follows by using Lemma \ref{lemma: mu_n_id_derivative}. A simple argument using Leibniz' rule shows that it suffices to prove this claim for $r = s$. We start with the identity
\begin{equation*}
\frac{1}{(1+z)^{nr}} = \frac{1}{\Gamma(rn)} \int_{0}^{\infty} e^{-(1+z)t} t^{rn - 1} dt, \quad z>0.
\end{equation*}
Hence, after the substitution $y = tz$, we obtain
\begin{align*}
\frac{d^{k}}{dz^{k}} \left(\frac{z^{n(r+1)}}{(1+z)^{nr}}\right) &= \frac{d^k}{dz^k} \left(\frac{z^n}{\Gamma(rn)} \int_0^{\infty} e^{-y(1 + \frac{1}{z})} y^{rn-1} dy\right) \\
&= \frac{1}{\Gamma(rn)} \int_0^{\infty} \frac{d^{k}}{dz^{k}} \left(z^n e^{-\frac{y}{z}}\right) e^{-y} y^{rn-1} dy \\
&= \frac{1}{\Gamma(rn)} \int_0^{\infty} \frac{d^{k}}{du^{k}} \left.\left(u^n e^{-\frac{1}{u}}\right)\right|_{u = \frac{z}{y}} y^{n-k} e^{-y} y^{rn-1} dy.
\end{align*}
Using \cite[Ex. 73 p. 388]{Szego_Orth_Pol} and the sum representation for Laguerre polynomials (see e.g. \cite[Formula 5.1.6]{Szego_Orth_Pol}) 
\begin{equation*}
L_k^{\alpha}(x) = \sum_{j = 0}^k (-1)^j {k + \alpha \choose k-j} \frac{x^j}{j!},
\end{equation*}
this can be rewritten as
\begin{equation*}
\frac{k!}{\Gamma(rn)} \int_0^{\infty} e^{-y(1 + \frac{1}{z})} y^{rn-1} z^{n-k} \left(\sum_{j = 0}^k {k-n-1 \choose k-j} \frac{(-1)^{k-j}}{j!} \left(\frac{y}{z}\right)^j\right) dy.
\end{equation*}
Since the sign of ${k-n-1 \choose k-j}$ is given by \((-1)^{k-j}\) the claim follows. 
\end{proof}

We are now ready to state our main result of this section. First, we define
\begin{equation} \label{eq: w_ast}
w^{\ast} := \frac{a(r + 1 - s) - r + \sqrt{(a(r + 1 - s) - r)^2 + 4a(r+1)(r-s)}}{2(r-s)}
\end{equation}
and
\begin{equation} \label{eq: x_ast}
x^{\ast} := \frac{r+1}{s+1} \frac{(w^{\ast})^r}{(w^{\ast} + 1)^{s-1}\left(w^{\ast} - \frac{as - 1}{s+1}\right)}.
\end{equation}
These quantities are derived in Proposition \ref{prop: study_alg_eq}.
\begin{theorem} \label{thm: measure_mu}
Let $r,s \in \mathbb{N}$ such that $s < r$ and let $a$ be a positive real number. Then there exists a unique measure $J_{r,s,a}$ on $[0, x^{\ast}]$ with total mass $a$ such that the moments are given by the numbers \eqref{eq: def_mu_0} and \eqref{eq: def_mu_n}. 
\end{theorem}
\begin{proof}
First, notice that $x^{\ast} > 0$ for all $a > 0$ and $s < r$. Indeed, one can easily check that $w^{\ast} > a$ and thus, since $\frac{as-1}{s+1} < a$, we have that $x^{\ast} > 0$. 

Consider now the algebraic equation
\begin{equation} \label{eq: algebraic_eq_w}
w^{r+1} - x(w - a)(w+1)^s = 0.
\end{equation}
As we will show in Proposition \ref{prop: study_alg_eq} this algebraic equation has a unique solution $w(x)$ which is analytic at infinity such that $w(x) \to a$ as $x \to \infty$. Furthermore, this solution has an analytic continuation to $\mathbb{C}\setminus[0, x^{\ast}]$. Since $w(x)$ is a solution of \eqref{eq: algebraic_eq_w} we know that
\begin{equation*}
w(x) = a + \frac{1}{x} \frac{w(x)^{r+1}}{(w(x)+1)^s}
\end{equation*}
and now applying the Lagrange-B\"urmann theorem gives us that
\begin{equation*}
w(x) = a + \sum_{n = 1}^{\infty} \frac{1}{n!} \frac{d^{n-1}}{dz^{n-1}} \left.\left(\frac{z^{n(r+1)}}{(1+z)^{ns}}\right)\right|_{z = a} x^{-n}
\end{equation*}
in a neighbourhood of infinity. Because of Lemma \ref{lemma: mu_n_id_derivative} this can be rewritten as
\begin{equation} \label{eq: w_moments_series}
w(x) = \sum_{n = 0}^{\infty} J_{r,s,a}(n) x^{-n}.
\end{equation}
We define now
\begin{equation} \label{eq: def_density}
\rho(x) := \frac{1}{2\pi i} \left\lbrace\frac{w_{-}(x)}{x} - \frac{w_+(x)}{x}\right\rbrace, \qquad x \in (0, x^{\ast})
\end{equation}
where $w_-(x)$, resp. $w_+(x)$, denotes the limiting value of $w(z)$ as $z$ approaches $x$ with $\text{Im}(z) < 0$, resp. $\text{Im}(z) > 0$. First of all, we notice that $w_-(x) = \overline{w_+(x)}$ so that $\rho(x)$ is a real-valued function. Furthermore, we claim that $\rho(x)$ is an integrable, everywhere positive function such that
\begin{equation*}
\int_0^{x^{\ast}} x^n \rho(x) dx = J_{r,s,a}(n) \qquad \text{for all } n \in \mathbb{N}
\end{equation*}
and thus is the density of a (unique) measure on $[0, x^{\ast}]$ with the numbers $J_{r,s,a}(n)$ as moments. 

By standard arguments and using $w(0) = 0$ one can see that
\begin{equation*}
\int_0^{x^{\ast}} x^n \rho(x) dx = \oint_{K} z^n \frac{w(z)}{z} dz = \oint_{K} z^{n-1} w(z) dz
\end{equation*}
with $K$ a positively oriented, closed contour encircling the cut $[0, x^{\ast}]$. Since $z^{n-1} w(z)$ has no singularities in $\mathbb{C}\setminus [0, x^{\ast}]$, we can can compute the residue at infinity and use \eqref{eq: w_moments_series} to obtain that
\begin{equation*}
\oint_{K} z^{n-1} w(z) dz = J_{r,s,a}(n)
\end{equation*}
and hence
\begin{equation*}
\int_0^{x^{\ast}} x^n \rho(x) dx = J_{r,s,a}(n).
\end{equation*}
Finally, to prove that $\rho(x) > 0$ for all $x \in (0, x^{\ast})$, it now suffices to show that $w_-(x) \neq w_+(x)$ if $x \in (0, x^{\ast})$. Indeed, if $w_-(x) \neq w_+(x)$, then $\rho(x) \neq 0$ and thus, because of the continuity, $\rho(x)$ is either everywhere positive or everywhere negative. Since
\begin{equation*}
\int_0^{x^{\ast}} \rho(x) dx = J_{r,s,a}(0) = a
\end{equation*}
we obtain that $\rho(x)$ must be positive-valued. 

So let $x \in (0, x^{\ast})$ and let $K_x$ be a positively oriented circle around the origin with radius $x$ starting at the point $x$. Then we have
\begin{equation*}
1 = \frac{1}{2\pi i} \oint_{K_x} \frac{1}{z} dz = \frac{1}{2\pi i} \oint_{K_x} \frac{1}{f(w(z))} dz
\end{equation*}
where we used the fact that
\begin{equation*}
z = \frac{w(z)^{r+1}}{(w(z) - a)(w(z)+1)^s} =: f(w(z)).
\end{equation*}
Making $u = w(z)$ the new variable of integration and observing that $1 = f'(w(z)) w'(z)$ and
\begin{equation*}
f'(w) = \left(\frac{r+1}{w} - \frac{1}{w-a} - \frac{s}{w+1}\right)f(w)
\end{equation*}
we get that
\begin{equation*}
1 = \frac{1}{2\pi i} \oint_{w(K_x)} \left(\frac{r+1}{u} - \frac{1}{u-a} - \frac{s}{u+1}\right) du.
\end{equation*}
Here $w(K_x)$ is a contour starting at $w_+(x)$ and ending at $w_-(x)$. 
Assume now that $w_-(x) = w_+(x)$, i.e. $w_-(x) = w_+(x)$ is a real number. Because \eqref{eq: algebraic_eq_w} cannot have nonnegative solutions for $w$ if $x \in (0, x^{\ast})$ we see that $w(K_x)$ will be a closed contour starting in a point on the negative axis, going to the complex plane, crossing the real axis exactly one more time between the origin and $a$ and returning to its starting point. The orientation can be positive or negative, and the contour can encircle $-1$ but not $a$. Thus we get the following possibilities
\begin{equation*}
\frac{1}{2\pi i} \oint_{w(K_x)} \left(\frac{r+1}{u} - \frac{1}{u-a} - \frac{s}{u+1}\right) du = \begin{cases}r+1 \\ r+1-s \\ -(r+1) \\ -(r+1-s)\end{cases} \neq 1.
\end{equation*}
We see that in all cases we get a contradiction and thus $w_-(x) \neq w_+(x)$. Hence we can conclude that \eqref{eq: def_density} indeed defines a density.
\end{proof}

We end this section with the following proposition, which was needed in our proof of Theorem \ref{thm: measure_mu}. 
\begin{proposition} \label{prop: study_alg_eq}
Let $r, s$ be positive integers such that $s < r$ and let $a > 0$. Then the equation
\begin{equation} \label{eq: algebraic_eq_w_2}
w^{r+1} - x(w - a)(w+1)^s = 0
\end{equation}
defines an algebraic function $w(x)$ which has an analytic branch at infinity with $w(x) \to a$, as $x \to \infty$. Moreover, this branch admits an analytic continuation to $\mathbb{C} \setminus [0, x^{\ast}]$ where $x^{\ast}$ is given by \eqref{eq: x_ast}.
\end{proposition}
\begin{proof}
The existence of a solution which is analytic at infinity can be seen by rewriting \eqref{eq: algebraic_eq_w_2} as
\begin{equation*}
w = a + \frac{1}{x} \frac{w^{r+1}}{(w+1)^s},
\end{equation*}
which permits one to apply Lagrange-B\"urmann's theorem. We confine this solution to the first sheet of the Riemann surface associated to \eqref{eq: algebraic_eq_w_2} and we denote it by $w_1(x)$. It remains to prove that this solution has an analytic continuation to $\mathbb{C} \setminus [0, x^{\ast}]$. To this end, we want to find all the branch points and thus we have to solve the following system of equations in the variables $x$ and $w$
\begin{equation} \label{eq: solve_system}
\begin{cases}w^{r+1} - x(w-a)(w+1)^s = 0 \\(r+1)w^r - x(w+1)^{s-1}((s+1)w - (as-1)) = 0\end{cases}.
\end{equation}
One can now immediately see that $x = 0$ is a branch point with $w = 0$. This is a multiple branch point connecting all the $r+1$ sheets of the associated Riemann surface. Moreover, in the case $a = \frac{1}{s}$, the second equation of \eqref{eq: solve_system} gives us that $w = 0$ and thus $x = 0$. So from now on we assume that $a \neq \frac{1}{s}$. Then the second equation can be rewritten as
\begin{equation*}
x = \frac{r+1}{s+1} \frac{w^r}{(w+1)^{s-1} \left(w - \frac{as-1}{s+1}\right)}.
\end{equation*}
Substituting this in the first equation gives us
\begin{equation*}
w^{r+1} - \frac{r+1}{s+1} \frac{w^r}{(w+1)^{s-1} \left(w - \frac{as-1}{s+1}\right)} (w-a)(w+1)^s = 0.
\end{equation*}
Assuming that $w \neq 0$, this can be simplified to
\begin{equation*}
w((s+1)w - (as-1)) - (r+1)(w-a)(w+1) = 0. 
\end{equation*}
One can now easily check that the two solutions of this quadratic equation are given by  $w = w^{\ast}$ where $w^{\star}$ is defined in \eqref{eq: w_ast} and by $w = \tilde{w}$ with
\begin{equation} \label{eq: w_tilde}
\tilde{w} := \frac{a(r + 1 - s) - r - \sqrt{(a(r + 1 - s) - r)^2 + 4a(r+1)(r-s)}}{2(r-s)}.
\end{equation}
Hence, we can conclude that the only possible branch points are at the real points $x = 0$, $x = x^{\ast}$, $x = \tilde{x}$ with
\begin{equation} \label{eq: x_tilde}
\tilde{x} := \frac{r+1}{s+1} \frac{(\tilde{w})^r}{(\tilde{w} + 1)^{s-1}\left(\tilde{w} - \frac{as - 1}{s+1}\right)}
\end{equation}
and at infinity. 

To conclude that on the first sheet $w(x)$ only has branch points at $x = 0$ and at $x = x^{\ast}$, it now suffices to show that there is no branch point at $x = \tilde{x}$ on the first sheet. Indeed, due to the analyticity of $w_1(x)$ there cannot be a branch point at infinity on this sheet. From equation \eqref{eq: algebraic_eq_w_2} it can be observed that $w_1(x)$ admits an analytic continuation starting at infinity travelling along the negative real axis up to the origin. In the same manner, $w_1(x)$ can be analytically continued starting at infinity travelling along the positive real axis up to $x = x^{\ast}$. Moreover, we have $w_1(x) > 0$ on $\mathbb{R}\setminus [0, x^{\ast}]$. As all branch points are real, we can conclude that $w_1(x)$ admits an analytic continuation onto $\mathbb{C}\setminus [0, x^{\ast}]$. Taking into account that $\tilde{x} \notin [0, x^{\ast}]$ and $w(\tilde{x}) = \tilde{w} < 0$, this shows that there can be no further branch point on the first sheet which completes the proof.
\end{proof}

\begin{remark}
Using the observations made in the proof of Proposition \ref{prop: study_alg_eq} we can now describe the geometry of the Riemann surface associated to the algebraic equation \eqref{eq: algebraic_eq_w_2}. We illustrate this in Figure \ref{fig: riemann_surface} for the case $r = 5, s=3$.
\end{remark}

\begin{figure}[h!] 
\centering
\begin{tikzpicture}
\draw (0,0)--(1,1)--(4,1)--(3,0)--cycle;															\draw (8,0)--(9,1)--(12,1)--(11,0)--cycle;
\draw [red] (2, 0.5)--(2.8,0.5);																		\draw [Orange] (8.5,0.5)--(10,0.5);							\draw [Orange, dashed] (10,0.5)--(10,-3.5); \draw [Orange, dashed] (8.5,0.5)--(8.5,-3.5);
\draw [red, dashed] (2,0.5)--(2,-1.5);													\draw [Orange, dashed] (8.5,0.5)--(5.5,-1.5);
\draw [red, dashed] (2.8,0.5)--(2.8,-1.5);									\draw [Orange, dashed] (10,0.5)--(7,-1.5);
\draw [red] (2, -1.5)--(2.8,-1.5); 											\draw [Orange] (5.5,-1.5)--(7,-1.5);
\draw (0,-2)--(1,-1)--(4,-1)--(3,-2)--cycle; 							\draw (5,-2)--(6,-1)--(9,-1)--(8,-2)--cycle;
\draw [ForestGreen] (0.5,-1.5)--(2,-1.5);							\draw [Blue] (7,-1.5)--(8,-1.5);
\draw [ForestGreen, dashed] (0.5,-1.5)--(0.5,-3.5);			\draw [Blue, dashed] (7,-1.5)--(2,-3.5);
\draw [ForestGreen, dashed] (2,-1.5)--(2,-3.5);				\draw [Blue, dashed] (8,-1.5)--(3,-3.5);
\draw [ForestGreen] (0.5,-3.5)--(2,-3.5);							\draw [Blue] (2,-3.5)--(3,-3.5);		
																										\draw [Orange, dashed] (8.5,-3.5)--(5.5,-1.5);
																												\draw [Orange, dashed] (10,-3.5)--(7,-1.5);	
																														\draw [Orange] (8.5,-3.5)--(10,-3.5);
\draw (0,-4)--(1,-3)--(4,-3)--(3,-4)--cycle;													\draw (8,-4)--(9,-3)--(12,-3)--(11,-4)--cycle;

\node [above right, red] at (2.8,0.5) {$x^{\ast}$};
\fill[red] (2.8,0.5) circle (2pt);
\node [above left] at (2,0.5) {$0$};
\fill (2,0.5) circle (2pt);
\node [above right, red] at (2.8,-1.5) {$x^{\ast}$};
\fill[red] (2.8,-1.5) circle (2pt);
\fill (2,-1.5) circle (2pt);
\fill (2,-3.5) circle (2pt);
\node [above right, Blue] at (3,-3.5) {$\tilde{x}$};
\fill[Blue] (3,-3.5) circle (2pt);
\node [above right, Blue] at (8,-1.5) {$\tilde{x}$};
\fill[Blue] (8,-1.5) circle (2pt);
\fill (7,-1.5) circle (2pt);
\fill (10,0.5) circle (2pt);
\fill (10,-3.5) circle (2pt);
\end{tikzpicture}
\caption{The geometry of the Riemann surface associated to \eqref{eq: algebraic_eq_w_2} for $r=5, s=3$.}
\label{fig: riemann_surface}
\end{figure}
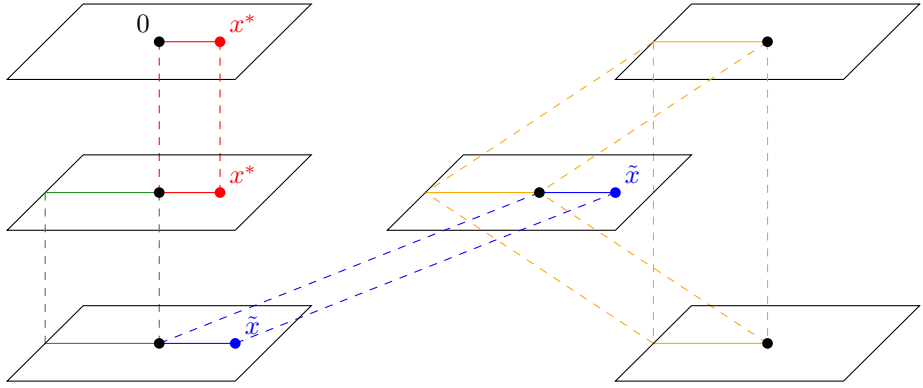

\section{Application to random matrix theory} \label{sec: appl_rmt}

In this section we will show how the measures obtained in Theorem \ref{thm: measure_mu} arise naturally in random matrix theory and free probability. We start with a small introduction in free probability theory which we need to state our second theorem. For more details, we refer the reader to \cite{Voiculescu_DN}, \cite{Anderson_GZ}, \cite{Speicher} or \cite{Akemann_Handbook}. 

\subsection{Free probability and random matrices}

Given a (compactly supported) probability measure $\mu$ on $\mathbb{R}$ such that $\int_{\mathbb{R}}x d\mu(x) \neq 0$ we define its $S$-transform as follows. Let $G_{\mu}$ denote the Stieltjes transform of the measure $\mu$, i.e.
\begin{equation*} \label{eq: def_stieltjes}
G_{\mu}(z) := \int_{\mathbb{R}} \frac{1}{z-x} d\mu(x), \qquad z \in \mathbb{C}\setminus\text{supp}(\mu)
\end{equation*}
and define 
\begin{equation} \label{eq: def_psi_mu}
\psi_{\mu}(z) := \frac{1}{z} G_{\mu}\left(\frac{1}{z}\right) - 1.
\end{equation}
Let $\chi_{\mu}$ be the unique function, analytic in a neighbourhood of zero, satisfying
\begin{equation} \label{eq: def_chi_mu}
\chi_{\mu}(\psi_{\mu}(z)) = z.
\end{equation}
Then the $S$-transform, denoted by $S_{\mu}$, is defined as
\begin{equation} \label{eq: def_s_tr}
S_{\mu}(z) := \frac{z+1}{z} \chi_{\mu}(z).
\end{equation}
Given two (compactly supported) probability measures $\mu$ and $\nu$ with non-vanishing first moments, the free multiplicative convolution, denoted by $\mu \boxtimes \nu$, is the unique (compactly supported) probability measure that satisfies the identity
\begin{equation}
S_{\mu \boxtimes \nu}(z) = S_{\mu}(z) S_{\nu}(z).
\end{equation}
Notice that this identity shows us that the free multiplicative convolution is commutative, i.e. $\mu \boxtimes \nu = \nu \boxtimes \mu$. The $S$-transform is an important tool in free probability theory to compute the distribution of, for instance, the product of free random variables.

By $\mu_A$ we denote the empirical eigenvalue distribution of an $n \times n$ random matrix $A$, i.e.
\begin{equation} 
\mu_A = \frac{1}{n} \sum_{i = 1}^n \delta_{\lambda_i(A)}
\end{equation}
with $\lambda_i(A)$ the $n$ random eigenvalues of $A$. 

With this in mind, we now have the following result (\cite{Couillet_Debbah}, Theorem 4.7, p. 82, see also \cite{Voiculescu} and \cite{Akemann_Handbook}).

\begin{theorem} \label{thm: distr_prod_matrices}
Let $\{A_n\}$ and $\{B_n\}$ be two sequences of random matrices of size $n \times n$ such that $A_n > 0$, i.e. all eigenvalues are positive (with probability 1), and such that $A_n$ and $B_n$ are asymptotically free almost surely for all $n$. Moreover, suppose that there exist two compactly supported probability measures $\mu_1$ and $\mu_2$ such that
\begin{equation}
\mu_{A_n} \stackrel{w}{\longrightarrow}  \mu_1 \qquad a.s. \qquad \text{and} \qquad \mu_{B_n} \stackrel{w}{\longrightarrow}  \mu_2 \qquad a.s.,
\end{equation} 
as $n \to \infty$, and where $\mu_{A_n}$ resp. $\mu_{B_n}$ denote the empirical eigenvalue distribution of $A_n$ resp. $B_n$. Then 
\begin{equation}
\mu_{A_nB_n} \stackrel{w}{\longrightarrow} \mu_1 \boxtimes \mu_2 \qquad a.s.,
\end{equation}
as $n \to \infty$.
\end{theorem}

Here, by
\begin{equation*}
\mu_{A_n} \stackrel{w}{\longrightarrow}  \mu_1 \qquad a.s.,
\end{equation*} 
as $n \to \infty$, we mean that 
\begin{equation*}
\int_{\mathbb{R}} f(t) d\mu_1(t) = \lim_{n \to \infty} \frac{1}{n} \sum_{i = 1}^n f(\lambda_i(A_n))
\end{equation*}
holds with probability $1$ for each bounded, continuous function $f(t)$. We say that $\mu_{A_n}$ converges weakly, almost surely to $\mu_1$.

\begin{remark}
Basically, Theorem \ref{thm: distr_prod_matrices} holds for all sequences $\{A_n\}$ and $\{B_n\}$ for which $A_n$ and $B_n$ are independent and for which at least the distribution of $A_n$ or $B_n$ is invariant under left and right multiplication by Haar distributed unitary matrices \cite{Couillet_Debbah, Akemann_Handbook, Voiculescu}. 
\end{remark}

\subsection{Product of Ginibre and truncated unitary matrices}

Let $U$ be a Haar distributed unitary random matrix of size $l \times l$ and let $T$ be the $m \times n$ upper left block of $U$ such that $l \geq m+n$. We call $T$ a truncated unitary (random) matrix of size $m \times n$ and the distribution is proportional to (see, e.g., \cite[Eq. (69)]{Fyodorov_Sommers})
\begin{equation*}
\det(I - T^{\ast}T)^{l - m-n} \chi_{T^{\ast}T \leq I}(T) dT
\end{equation*} 
where 
\begin{equation*}
\chi_{T^{\ast}T \leq I}(T) := \begin{cases}1 & \mbox{ if } I-T^{\ast}T \mbox{ is positive-definite} \\ 0 & \mbox{else}\end{cases}.
\end{equation*}

A complex Ginibre matrix $G$ of size $m \times n$ has independent entries whose real and
imaginary parts are independent and have a standard normal distribution with fixed variance. The probability distribution of $G$ is proportional to
\begin{equation*}
e^{- \Tr G^{\ast} G} dG.
\end{equation*}

We now take $s$ independent truncated unitary matrices $T_j$ of size $(n + \nu_j) \times (n + \nu_{j-1})$, with $\nu_j \geq 0$ and $\nu_0 = 0$, coming from an $l_j \times l_j$ unitary matrix. Furthermore, we take $r-s$ independent Ginibre random matrices $G_j$ of size $(n + \nu_j) \times (n + \nu_{j-1})$ for $j = s+1, \ldots, r$ and we define the product of independent matrices
\begin{equation} \label{eq: def_prod_yrs}
Y_{r,s} := G_r \ldots G_{s+1}T_s \ldots T_1.
\end{equation}
We then have the following theorem:

\begin{theorem} \label{thm: connection_rmt}
Let $T_j$ and $G_j$ be as described above. Furthermore, suppose that we have that $l_j - 2n\geq 0$ and $\nu_i$ remain fixed for all $j = 1, \ldots, s$ and all $i = 0, \ldots, r$, as $n \to \infty$. Then we have
\begin{equation}
\mu_{Y_n} \stackrel{w}{\longrightarrow} J_{r,s,1} \qquad a.s.,
\end{equation}
as $n \to \infty$, and where $J_{r,s,1}$ is described in Theorem \ref{thm: measure_mu} and $Y_n$ is defined as the rescaled Wishart-type product
\begin{equation}
Y_n := \frac{1}{n^{r-s}}Y_{r,s}^{\ast}Y_{r,s}.
\end{equation}
\end{theorem}

\begin{proof}
We start with the observation that $Y_n$ has the same non-zero eigenvalues (counted with multiplicity) as 
$$Z_n := \frac{1}{n^{r-s}} (G_r \ldots G_{s+1})^{\ast}(G_r \ldots G_{s+1})(T_s \ldots T_1)(T_s \ldots T_1)^{\ast}$$ 
Moreover, the difference in the number of eigenvalues equal to zero is $\nu_{s}$ and thus the limiting eigenvalue distributions of $Y_n$ and $Z_n$ have to be equal, if they exist. It is known that the empirical eigenvalue distribution of $\frac{1}{n^{r-s}}(G_r \ldots G_{s+1})^{\ast}(G_r \ldots G_{s+1})$ converges weakly, almost surely to the Fuss-Catalan distribution $FC_{r-s}$, as $n \to \infty$ (see, e.g., \cite{Penson_Z}), which can be written as the Raney distribution $R_{r-s+1,1}$. Moreover, the distribution of each $G_j$ is invariant under left and right multiplication of Haar distributed unitary matrices and thus the same holds true for 
$$\frac{1}{n^{r-s}}(G_r \ldots G_{s+1})^{\ast}(G_r \ldots G_{s+1}).$$ 
In order to apply Theorem \ref{thm: distr_prod_matrices} we now have to determine the limiting eigenvalue distribution of $\tilde{T}_n := (T_s \ldots T_1)(T_s \ldots T_1)^{\ast}$. 

By the same arguments as before, we know that $(T_s \ldots T_1)(T_s \ldots T_1)^{\ast}$ has the same non-zero eigenvalues (counted with multiplicity) as $(T_{s-1} \ldots T_1)(T_{s-1} \ldots T_1)^{\ast}T_s^{\ast}T_s$ and the difference in the number of eigenvalues equal to zero is $|\nu_s - \nu_{s-1}|$. It is known that the empirical eigenvalue distribution of $T_s^{\ast}T_s$ converges weakly, almost surely to the arcsine measure on $(0, 1)$ if $l_s - 2n$ is fixed, as $n \to \infty$. By comparing the moments, one can see that this is the Raney distribution $R_{1, \frac{1}{2}}$. Moreover, one can check that the distribution of $T_s$ is also invariant under left and right multiplication of Haar distributed unitary matrices and so is the distribution of $T_s^{\ast}T_s$. As before, to apply Theorem \ref{thm: distr_prod_matrices} we now have to determine the limiting eigenvalue distribution of $(T_{s-1} \ldots T_1)(T_{s-1} \ldots T_1)^{\ast}$. Repeating this argument $s-1$ times we can conclude that 
\begin{equation}
\mu_{\tilde{T}_n} \stackrel{w}{\longrightarrow} R_{1, \frac{1}{2}}^{\boxtimes s} \qquad a.s.,
\end{equation}
as $n \to \infty$, and with $\tilde{T}_n := (T_s \ldots T_1)(T_s \ldots T_1)^{\ast}$. 

An application of Theorem \ref{thm: distr_prod_matrices} gives us that
\begin{equation*}
\mu_{Z_n} \stackrel{w}{\longrightarrow} \kappa \qquad a.s.,
\end{equation*}
as $n \to \infty$, where
\begin{equation} \label{eq: def_kappa}
\kappa := R_{r-s+1, 1} \boxtimes R_{1, \frac{1}{2}}^{\boxtimes s}.
\end{equation}

It remains to show that $\kappa = J_{r,s,1}$. Using the result from Mlotkowski \cite[Proposition 4.3]{Mlotkowski}, we know that
\begin{equation*}
S_{R_{r-s+1}, 1}(z) = \frac{1}{(1+z)^{r-s}}, \qquad S_{R_{1, \frac{1}{2}}}(z) = \frac{z+2}{z+1}
\end{equation*}
and hence
\begin{equation*}
S_{\kappa}(z) = \frac{(z+2)^s}{(z+1)^r}.
\end{equation*}
Using \eqref{eq: def_s_tr} this can be rewritten as
\begin{equation*}
\chi_{\kappa}(z) = \frac{z(z+2)^s}{(z+1)^{r+1}}.
\end{equation*}
Thus, if we replace $z$ by $\psi_{\kappa}(z)$ and use \eqref{eq: def_chi_mu} we obtain
\begin{equation*}
z = \frac{\psi_{\kappa}(z) (\psi_{\kappa}(z) + 2)^s}{(\psi_{\kappa}(z) + 1)^{r+1}}.
\end{equation*}
Finally applying identity \eqref{eq: def_psi_mu} and replacing $z$ by $1/z$ we arrive at
\begin{equation*}
\frac{1}{z} = \frac{(zG_{\kappa}(z) - 1)(z G_{\kappa}(z) + 1)^s}{(z G_{\kappa}z)^{r+1}}
\end{equation*}
and from this we can conclude that $w(x) = x G_{\kappa}(x)$ satisfies the algebraic equation
\begin{equation*}
w(x)^{r+1} - x (w(x) - 1)(w(x) + 1)^s = 0.
\end{equation*}
This is equation \eqref{eq: algebraic_eq_w} with $a = 1$ and thus, because of Theorem \ref{thm: measure_mu}, we obtain 
\begin{equation} \label{eq: kappa_equal_mu1}
\kappa = J_{r,s,1}.
\end{equation} 
\end{proof}

\begin{remark} \label{remark: lim_distr_case_s_1}
Theorem \ref{thm: connection_rmt} in combination with equations \eqref{eq: def_kappa} and \eqref{eq: kappa_equal_mu1} gives us that 
\begin{equation*}
J_{r,0, 1} = FC_r, \qquad J_{r,1,1} = R_{\frac{r+1}{2}, \frac{1}{2}}.
\end{equation*}
This can be seen immediately by using
\begin{equation*}
R_{r, 1} \boxtimes R_{1, \frac{1}{2}} = R_{\frac{r+1}{2}, \frac{1}{2}},
\end{equation*}
which is a special case of the identity stated in \cite[Proposition 4.3]{Mlotkowski}. This means in particular that
\begin{equation*}
J_{r, 0, 1}(n) = FC_r(n), \qquad J_{r, 1, 1}(n) = R_{\frac{r+1}{2}, \frac{1}{2}}(n)
\end{equation*}
for every $n \in \mathbb{N}$. It is interesting to remark that for these distributions explicit and elementary forms of the densities can be found by the method of parametrization (see, e.g., \cite{Forrester_Liu, MNPZ, Neuschel, Neuschel_Stivigny}).
\end{remark}

\begin{remark}
The statements of Theorem \ref{thm: measure_mu} and Theorem \ref{thm: connection_rmt} are restricted to the case $s < r$ because of several technical issues that arise in the case $r = s$. However, it is natural and interesting to ask whether our results can be extended to this case.
\end{remark}

\section*{Acknowledgments}

We thank Prof. Arno Kuijlaars for many valuable discussions. The last two authors are supported by KU Leuven Research Grant OT/12/073 and the Belgian Interuniversity Attraction Pole P07/18.


\begin{thebibliography}{99}

\bibitem{Akemann_Nagao} G. Akemann, Z. Burda, M. Kieburg and T. Nagao,
Universal microscopic correlation functions for products of truncated unitary matrices,
preprint arXiv: 1310.6395.

\bibitem{Akemann_Kieburg_Wei} G. Akemann, M. Kieburg, and L. Wei,
Singular value correlation functions for products of {W}ishart random matrices,
\emph{J. Phys. A} 46 (2013), 275205, 22 pp.


\bibitem{Alexeev_Gotze} N. Alexeev, F. G{\"o}tze and A. Tikhomirov,
Asymptotic distribution of singular values of powers of random matrices,
\emph{Lith. Math. J.} 50 (2010), no. 2, 121--132.

\bibitem{Anderson_GZ} G.W. Anderson,  A. Guionnet and O. Zeitouni,
An introduction to random matrices,
Cambridge Studies in Advanced Mathematics, Vol. 118,
Cambridge University Press, Cambridge, 2010.

\bibitem{Bougerol_Lacroix} P. Bougerol and J. Lacroix,
Products of random matrices with applications to {S}chr\"odinger operators,
Progress in Probability and Statistics, 8 (1985), xii+283.

\bibitem{Burda_Janik_Waclaw} Z. Burda, R.A. Janik and B. Waclaw,
Spectrum of the product of independent random {G}aussian matrices,
\emph{Phys. Rev. E (3)} 81 (2010), no. 4, 041132.

\bibitem{Burda_1} Z. Burda, A. Jarosz, G. Livan, M.A. Nowak, and A. Swiech,
Eigenvalues and singular values of products of rectangular {G}aussian random matrices,
\emph{Phys. Rev. E (3)} 82 (2010), no. 6, 061114.

\bibitem{Burda_extended} Z. Burda, A. Jarosz, G. Livan, M.A. Nowak, and A. Swiech,
Eigenvalues and singular values of products of rectangular {G}aussian random matrices --- the extended version,
\emph{Acta Phys. Polon. B} 42 (2011), no. 5, 939--985.

\bibitem{Couillet_Debbah} R. Couillet and M. Debbah,
Random matrix methods for wireless communications,
Cambridge University Press, Cambridge, 2011.

\bibitem{Dupic} T. Dupic and I. Isaac P\'{e}rez Castillo,
Spectral density of products of Wishart dilute random matrices. Part I: the dense case,
preprint arXiv:1401.7802.

\bibitem{Forrester} P.J. Forrester,
Eigenvalue statistics for product complex Wishart matrices, 
preprint arXiv: 1401.2572.

\bibitem{Forrester_Liu} P.J. Forrester and D.-Z. Liu, 
Raney distributions and random matrix theory, 
preprint arXiv: 1404:5759v1.

\bibitem{Furstenberg_Kesten} H. Furstenberg and H. Kesten,
Products of random matrices,
\emph{Ann. Math. Statist.} 31 (1960), 457--469.

\bibitem{Fyodorov_Sommers} Y.V. Fyodorov and H.-J. Sommers, 
Random matrices close to {H}ermitian or unitary: overview of methods and results,
\emph{J. Phys. A} 36 (2003), no. 12, 3303--3347.

\bibitem{Gotze_Tikhomirov} F. G\"otze and A. Tikhomirov,
On the asymptotic spectrum of products of independent random matrices,
preprint arXiv:1012.2710.

\bibitem{Knuth} R.L. Graham, D.E. Knuth and O. Patashnik, 
Concrete mathematics,
Addison-Wesley Publishing Company, Reading, MA, 2nd edition, 1994.


\bibitem{Kuijlaars_Stivigny} A. Kuijlaars, D. Stivigny, 
Singular values of products of random matrices, 
Random Matrices: Theory and Applications (2014) DOI 10.1142/S2010326314500117.

\bibitem{Mlotkowski} W. Mlotkowski,
Fuss-Catalan numbers in noncommutative probability,
\emph{Documenta Mathematica} 15 (2010), 939--955.

\bibitem{MNPZ} W. Mlotkowski, M.A. Nowak, K.A. Penson and K. Zycskowski,
Spectral density of generalized Wishart matrices and free multiplicative convolution,
preprint arXiv: 1407.1282.

\bibitem{Neuschel} T. Neuschel,
Plancherel-Rotach formulae for average characteristic polynomials of products of Ginibre random matrices and the Fuss-Catalan distribution,
\emph{Random Matrices: Theory and Appl.} 03 (2014), no. 01, 1450003.

\bibitem{Neuschel_Stivigny} T. Neuschel and D. Stivigny
Asymptotics for characteristic polynomials of Wishart type products of complex Gaussian and truncated unitary random matrices,
preprint arXiv:1407.2755.

\bibitem{ORourke_Soshnikov} S. O'Rourke and A. Soshnikov, 
Products of independent non-{H}ermitian random matrices,
\emph{Electron. J. Probab.} 16 (2011), no. 81, 2219--2245.

\bibitem{Penson_Z} K.A. Penson and K. {\.Z}yczkowski,
Product of {G}inibre matrices: {F}uss-{C}atalan and {R}aney distributions,
\emph{Phys. Rev. E} 83 (2011), 061118, 9 pp.

\bibitem{Speicher} R. Speicher,
Free probability and random matrices,
preprint arXiv: 1404.3393.

\bibitem{Akemann_Handbook} R. Speicher, 
Free Probability Theory, \emph{Chapter 22 of The Oxford handbook of random matrix theory},
Oxford Univ. Press, Oxford, 2011.

\bibitem{Szego_Orth_Pol} G. Szeg{\H{o}},
Orthogonal polynomials,
American Mathematical Society, Colloquium Publications, Vol. XXIII,
American Mathematical Society, Providence, R.I., 1975.

\bibitem{Tulino_Verdu} A.M. Tulino and S. Verd\'{u},
Random Matrix Theory and Wireless Communications,
\emph{Commun. Inf. Theory} 1 (2004), no. 1, 1--182.

\bibitem{Voiculescu} D.V. Voiculescu,
Limit laws for random matrices and free products,
\emph{Invent. Math.} 104 (1991), no. 1, 201--220.

\bibitem{Voiculescu_DN} D.V. Voiculescu, K.J. Dykema and A. Nica,
Free random variables,
CRM Monograph Series, Vol. 1,
American Mathematical Society, Providence, RI, 1992.
\end{thebibliography}
\end{document}